\documentclass{amsart}
\usepackage{bm,color}

\title
[Abstract formulation of the Cole-Hopf transform]
{
Abstract formulation \\ 
of the Cole-Hopf transform
}

\author{Yoritaka Iwata}

\address[Yoritaka Iwata]{ 
Institute of Innovative Research, Tokyo Institute of Technology;  \linebreak
Department of Mathematics, Shibaura Institute of Technology.}
\email{iwata\_phys@08.alumni.u-tokyo.ac.jp}

\thanks{The author is grateful to Prof. Emeritus Hiroki Tanabe for valuable comments.}

\theoremstyle{plain}
\newtheorem{theorem}{Theorem}

\newtheorem{lemma}[theorem]{Lemma}

\begin{document}

\begin{abstract}
Operator representation of Cole-Hopf transform is obtained based on the logarithmic representation of infinitesimal generators.
For this purpose the relativistic formulation of abstract evolution equation is introduced.
Even independent of the spatial dimension, the Cole-Hopf transform is generalized to a transform between linear and nonlinear equations defined in Banach spaces.
In conclusion a role of transform between the evolution operator and its infinitesimal generator is understood in the context of generating nonlinear semigroup.
\end{abstract}

\maketitle
\section{Introduction}  \label{sec1}
The Cole-Hopf transform \cite{15bateman, 48burgers, 51cole, 06forsyth, 50hopf} is a concept bridging the linearity and the nonlinearity.
For $t \in {\mathbb R}_+$ and $x \in {\mathbb R}$ the transformation reads
\begin{equation}  \begin{array}{ll}
 \psi(t,x) = -2 \mu^{-1/2} ~ \partial_x \log u(t,x), 
\end{array} \end{equation}
where $u(t,x)$ denotes the solution of linear equation, and $\psi(t,x)$ is the solution of transformed nonlinear equation.
On the other hand, e.g., for $t \in {\mathbb R}$ and $x \in {\mathbb R}^n$, the logarithmic representation of infinitesimal generator
\begin{equation} \begin{array}{ll}
 A(t) u_s(x) =  (I+ \kappa U(s,t))~ \partial_{t} {\rm Log} ~ (U(t,s) + \kappa I) u_s(x)
\end{array} \end{equation}
has been obtained in the abstract framework~\cite{17iwata-1}, where $U(t,s)$ denotes the evolution operator, and $A(t)$ is its infinitesimal generator.  
By taking a specific case with $\kappa = 0$, the similarity between them is clear.
That is, the process of obtaining infinitesimal generators from evolution operators is expected to be related to the emergence of nonlinearity.

In terms of the research history for the logarithm of operators, the logarithm of an injective sectorial operator was initiated by Nollau~\cite{69nollau} in 1969.
After some notable progresses (for example, see \cite{boyadzhiev, 03hasse, 06hasse, martinez, 00okazawa-01, 00okazawa-02}), the logarithm of operator is applied to represent the infinitesimal generator recently by Ref.~\cite{17iwata-1}. 
Following that work, the algebraic background of logarithmic representation of infinitesimal generators is studied \cite{17iwata-3}, where the product of unbounded infinitesimal generators in infinite-dimensional functional spaces is successfully defined by using the concept of a module over the Banach algebra.

In this paper the operator representation in abstract functional spaces is given to the Cole-Hopf transform.
Based on the logarithmic representation of infinitesimal generators obtained in Banach spaces, the Cole-Hopf transform is generalized in the following sense:
\begin{itemize}
\item the linear equation is not necessarily the heat equation;
\item the spatial dimension of the equations is not limited to 1;
\item the variable in the transform is not limited to a spatial variable $x$;
\end{itemize}
where, in order to realize these features, the relativistic formulation of abstract evolution equation is newly introduced.
Since the logarithmic representation shows a relation between an evolution operator and its infinitesimal generator, the correspondence to the Cole-Hopf transform means a possible appearance of nonlinearity in the process of defining an infinitesimal generator from the evolution operator.

\section{Mathematical settings}  \label{sec2}
\subsection{Logarithmic representation of infinitesimal generators}
Let $X$ and $B(X)$ be a Banach space with a norm $\| \cdot \|$ and a space of bounded linear operators on $X$, respectively.
The same notation is used for the norm equipped with $B(X)$, if there is no ambiguity.
For a positive and finite $T$, let $t$ and $s$ satisfy $-T \le t,s \le T$, $Y$ be a dense Banach subspace of $X$, and the topology of $Y$ be stronger than that of $X$.
The space $Y$ is assumed to be $U(t,s)$-invariant.
Let elements of evolution family $\{U(t,s) \}_{-T \le t, s  \le T}$ be mappings: $(t,s) \to U(t,s)$ satisfying the strong continuity for $-T \le t, s  \le T$.
The semigroup properties:
\begin{equation} \label{sg1} 
U(t,r)~ U(r,s) = U(t,s), 
\end{equation}
and
\begin{equation}  \label{sg2} 
U(s,s) = I, 
\end{equation}
are assumed to be satisfied, where $I$ denotes the identity operator of $X$.
Both $U(t,s)$ and $U(s,t)$ are assumed to be well-defined to satisfy
\begin{equation} \label{sg3}
U(s,t) ~ U(t,s) = U(s,s) = I,
\end{equation}
and the invertible property of $U(t,s)$ follows.
Consequently $U(t,s)$ is the two-parameter group on $X$. 

Next the pre-infinitesimal generator~\cite{17iwata-1} of $U(t,s)$ is introduced.
Pre-infinitesimal generator is a generalized concept of infinitesimal generator, but it is not necessarily an infinitesimal generator without assuming a dense property of domain space $Y$ in $X$.
In terms of extracting the universal feature independent of the domain of the operator, pre-infinitesimal generator is a useful concept for defining the logarithm of operator.
For $u_s \in Y$ and $-T \le t, s  \le T$, the pre-infinitesimal generator $A(t): Y  ~\to~  X$ is defined by
\begin{equation} \label{pe-group}
A(t) u_s := \mathop{\rm wlim}\limits_{h \to 0}  h^{-1} (U(t+h,t) - I) u_s
\end{equation}
using a weak limit, and then the definition of $t$-differential of $U(t,s)$ in a weak sense
\begin{equation} \label{de-group} \begin{array}{ll} 
\partial_t U(t,s)~u_s = A(t) U(t,s) ~u_s
\end{array}  \end{equation}
follows.
Accordingly the operator $A(t)$ is generally unbounded in $X$.
Equation~(\ref{de-group}) is regarded as the abstract evolution equation with $t$-dependent infinitesimal generator (for example, see \cite{79tanabe}).
Consequently, for $-T \le t \le T$, $A(t)$ is represented by means of the logarithm~\cite{17iwata-1}; there exists a certain complex number $\kappa \ne 0$ such that
\begin{equation} \label{logex} \begin{array}{ll}
 A(t) ~ u_s =  (I+ \kappa U(s,t))~ \partial_{t} {\rm Log} ~ (U(t,s) + \kappa I) ~ u_s,
\end{array} \end{equation}
where ${\rm Log}$ denotes the principal branch of logarithm.
A part $\partial_{t} {\rm Log} ~ (U(t,s) + \kappa I)$ itself is an infinitesimal generator, and the concept of evolution operators without satisfying the semigroup property follows \cite{17iwata-3}. 

\subsection{Relativistic formulation of abstract evolution equations}
Let the standard space-time variables $(t,x,y,z)$ be denoted by $(x^0,x^1,x^2,x^3)$ receptively.
It is further possible to generalize space-time variables to $(x^0,x^1,x^2,x^3, \cdots, x^{n})$ to be valid to $n$-dimensional space. 
In spite of the standard treatment of abstract evolution equations, the direction of evolution does not necessarily mean time-variable ``$t$'' in this relativistic treatment of the evolution equations.
For an evolution family of operators $\{U(x^i,\xi^i)\}_{-L \le x^i,\xi^i \le L}$ in a Banach space $X$, the condition to obtain the logarithmic representation is as follows~\cite{17iwata-1}: 
\begin{itemize}
\item $U(x^i,\xi^i)$ is the invertible evolution operator generated by $K(x^i)$;
\item $U(x^i,\xi^i)$ and $K(x^i)$ are assumed to commute. 
\end{itemize}
Note that commutation is necessarily satisfied if $K(x^i)$ is independent of $x^i$, so that the second condition does not matter for such cases. 
In the sense of relativistic formulation the standard abstract evolution equation (\ref{de-group}) is generalized to
\begin{equation}  \begin{array}{ll} 
\partial_{x^i} U(x^i, \xi^i)~u_{\xi^i} = K(x^i) U(x^i,\xi^i) ~u_{\xi^i}
 \end{array}  \end{equation}
in $X$.
The problem is that $K(x^i)$ is not necessarily an infinitesimal generator even if $K(x^j)$ is an infinitesimal generator for $i \ne j$.
The logarithmic representation becomes 
\begin{equation} \label{logex2} \begin{array}{ll}
 K(x^i) ~ u_{\xi^i} =
  (I+ \kappa U(\xi^i,x^i))~ \partial_{x^i} {\rm Log} ~ (U(x^i,\xi^i) + \kappa I) ~ u_{\xi^i},
\end{array} \end{equation}
where $u_{\xi}$ is an element in dense subspace $Y$ of $X$, and the notation follows from Eq.~(\ref{logex}).
If $x^i$ is not equal to $x^0$, then the direction of evolution can be naturally either positive or negative.
Note that, similar to the standard treatment of the abstract evolution equations, the boundary condition is usually included in the functional space.

\section{Cole-Hopf transform}  \label{sec3}
Before moving on to the main discussion developed in Banach spaces, the calculations related to the Cole-Hopf transform is reviewed.
The Cole-Hopf transform is the transform between the solutions of the Burgers' equation and those of heat equation. 
By changing the unknown functions as
\begin{equation}  \label{ch-trans} \begin{array}{ll}
 \psi(t,x) = -2 \mu^{-1/2} ~ \partial_x \log u(t,x)
  = -2 \mu^{-1/2} (\partial_x u(t,x)) ~ u(t,x)^{-1}, 
\end{array} \end{equation}
the linear heat equation 
\[ \begin{array}{ll}
 \partial_x^2 u -  {\mu}^{1/2}  {\partial_t} u = 0
\end{array} \]
is transformed to a nonlinear equation called the Burgers' equation
\begin{equation}  \label{bgeq} \begin{array}{ll}
\partial_t \psi + \psi \partial_x \psi =  \mu^{-1/2} ~ \partial_x^2 \psi,
\end{array} \end{equation}
where $u(t,x)$ permits its logarithm function (cf.~the hypoelliptic property and the positivity of parabolic evolution equation).
It is notable that this nonlinear equation is exactly solvable. 
The outline of the transformation is demonstrated in the following.
The derivatives are calculated by
\[ \begin{array}{ll}
\partial_t \psi(t,x)
 = - \frac{2}{\mu^{1/2}}  \Large\{ u(t,x)^{-1} \partial_t \partial_x u(t,x)
- u(t,x)^{-2}  \partial_t u(t,x)  \partial_x u(t,x) \Large\},     \vspace{2.5mm}\\
\partial_x \psi(t,x) 
 = - \frac{2} {\mu^{1/2}} ~ \Large\{ u(t,x)^{-1}   \partial_x^2 u(t,x)  - u(t,x)^{-2}  \left( \partial_x u(t,x) \right)^2  \Large\},      \vspace{2.5mm}\\
\partial_x^2 \psi(t,x) 
 = - \frac{2} {\mu^{1/2}}  \Large\{ u(t,x)^{-1}   \partial_x^3 u(t,x)  +2 u(t,x)^{-3}  \left( \partial_x u(t,x) \right)^3  \\
\hspace{18mm} -3 u(t,x)^{-2}   \partial_x u(t,x) ~ \partial_x^2 u(t,x)   \Large\} .
\end{array} \]
By substituting them, the Burgers' equation becomes
\[ \begin{array}{ll}
   \mu^{1/2}  \Large[  - u(t,x)^{-1}    \partial_t  \partial_x u(t,x)  
 + u(t,x)^{-2}  \partial_t u(t,x)  \partial_x u(t,x) \Large]  \\
 +   
  \Large[ 2 u(t,x)^{-2}  \partial_x u(t,x)  \partial_x^2 u(t,x)  
  - 2 u(t,x)^{-3}  \left( \partial_x u(t,x) \right)^3  \Large]  \\
 -\Large[ u(t,x)^{-1}   \partial_x^3 u(t,x)  - 3 u(t,x)^{-2}   \partial_x u(t,x) ~ \partial_x^2 u(t,x)   
 + 2 u(t,x)^{-3}  \left( \partial_x u(t,x) \right)^3
 \Large]
 = 0.
\end{array} \]
By multiplying $ u(t,x)^{2}$,
\[ \begin{array}{ll}
  u(t,x) \partial_x  [ \partial_t   u(t,x)-\mu^{-1/2}   \partial_x^2 u(t,x) ]   
  -\partial_x u(t,x) [  \partial_t   u(t,x)-\mu^{-1/2}   \partial_x^2 u(t,x) ]    = 0
\end{array} \]
follows.
If $ \partial_t   u(t,x)-\mu^{-1/2}   \partial_x^2 u(t,x) $ is nonzero, then
\begin{equation} \label{taisuu} \begin{array}{ll}
  \partial_x \log  \left[ \frac{\partial_t   u(t,x)-\mu^{-1/2}   \partial_x^2 u(t,x)}{u(t,x)} \right]   = 0
\end{array} \end{equation}
is true.
That is, for an arbitrary $f(t)$,
\[ \begin{array}{ll}
\partial_t   u(t,x)-\mu^{-1/2}   \partial_x^2 u(t,x)= f(t) u(t,x) 
\end{array} \]
is obtained, and it is equivalent to
\[ \begin{array}{ll}
\partial_t  [ u(t,x) \exp (-\int^t_{s} f d \tau) ]  
 -\mu^{-1/2}   \partial_x^2 [ u(t,x) \exp (-\int^t_{s} f d \tau) ] = 0. 
\end{array} \]
Due to the arbitrariness of $f(t)$, this equation is also regarded as the heat equation.
In this way the solutions of linear and nonlinear equations are connected by the Cole-Hopf transform.

\section{Main results}
\subsection{The abstract framework for the Cole-Hopf transform}
It is necessary to recognize the evolution direction of the heat equation as $x$, because the derivative on the spatial direction $x$ is considered in the Cole-Hopf transform.
Let us begin with one-dimensional heat equation
\begin{equation} \label{heat} \begin{array}{ll}
\partial_x^2 u(t,x) -\mu^{1/2}  \partial_t u(t,x) = 0, \quad t \in (0, \infty) , ~ x \in (-L, L), \vspace{2.5mm}\\
u(t,-L) = u(t,L)  = 0,  \quad t \in (0, \infty), \vspace{2.5mm}\\
u(0,x) = u_0(x),  \quad x \in (-L, L),
\end{array} \end{equation}
where $\mu$ is a real positive number, and the hypoelliptic property of parabolic evolution equation is true.
The first equation of (\ref{heat}) is hypoelliptic; for an open set ${\mathcal U} \subset (-\infty,\infty) \times  (-L,L)$, $u \in C^{\infty}({\mathcal U})$ follows from $(  \partial_x^2 - \mu^{1/2} \partial_t) u \in C^{\infty}({\mathcal U})$.

Equation~(\ref{heat}) is well-posed in $C^{\infty}(0,\infty) \times C^{\infty}(-L,L)$, so that $\mu^{1/2}  \partial_t$ is the infinitesimal generator in $C^{\infty}(-L,L)$.
The spaces $C^{1}(-L,L)$ and  $C^{\infty}(-L,L)$ are dense in $L^{p}(-L,L)$ with $1 \le p < \infty$.
The solution is represented by
\[ \begin{array}{ll}
 u(t,x) = U(t) u_0, \ 
\end{array} \]
where $U(t)$ is a semigroup generated by $\mu^{-1/2}  \partial_x^2$ under the Drichlet-zero boundary condition.

By changing the evolution direction from $t$ to $x$, the heat equation
\begin{equation} \label{heat2} \begin{array}{ll}
\partial_x^2 u(t,x) -\mu^{1/2}  \partial_t u(t,x) = 0, \quad t \in (0, \infty) , ~ x \in (-L, L), \vspace{2.5mm}\\
u(t,-L) = v_0(t), \quad \partial_x u(t,-L) = v_1(t),  \quad t \in (0, \infty), \vspace{2.5mm}\\
u(0,x) = 0,  \quad x \in (-L, L),
\end{array} \end{equation}
is considered for $x$-direction, where $v_0(t)$ and $v_1(t)$ are given initial functions.
To establish the existence of semigroup for the $x$-direction, it is sufficient to consider the generation of semigroup in $L^2(-\infty, \infty)$ by generalizing $t$-interval from $(0, \infty)$ to $(-\infty, \infty)$.
The Fourier transform leads to
\begin{equation}  \begin{array}{ll}
\partial_x^2 \tilde{u} - i \mu^{1/2} \omega \tilde{u}  = 0 \vspace{2.5mm}\\
\tilde{u} (\omega,0)  =  {\tilde v}_0(\omega)  , \quad  
\partial_x \tilde{u} (\omega,0) =  {\tilde v}_1(\omega) , 
\end{array} \label{eq-add} \end{equation}
where $\omega$ is a real number.
Indeed, the following transforms
\[  \begin{array}{ll}
u (t,x) = \frac{1}{2 \pi} \int_{-\infty}^{\infty}  \tilde{u}(\omega,x)  e^{ i \omega t} d \omega, \vspace{1.5mm} \\
u (t, -L) = \frac{1}{2 \pi}  \int_{-\infty}^{\infty} {\tilde u} (\omega, -L) e^{ i \omega t} d \omega    \vspace{1.5mm}, \\
\partial_x  u (t, -L) = \frac{1}{2 \pi}  \int_{-\infty}^{\infty}   \partial_x {\tilde u} (\omega, -L)  e^{i \omega t} d \omega  
\end{array} \]
are implemented.
By solving the characteristic equation $\lambda^2 -i \mu^{1/2} \omega = 0$, the Fourier transformed solution of (\ref{heat2}) is
\[  \begin{array}{ll}
\tilde{u}(\omega, x) 
 = \frac{ {\tilde v}_0(\omega)- i(- i \mu^{1/2}\omega )^{-1/2} {\tilde v}_1(\omega) }{2}  e^{ + ( i \mu^{1/2} \omega)^{1/2}  x}  
+  \frac{ {\tilde v}_0(\omega)+i  (- i \mu^{1/2}\omega)^{-1/2} {\tilde v}_1(\omega) }{2} e^{ - ( i \mu^{1/2} \omega)^{1/2}  x },
\end{array} \]
where
\[  \begin{array}{ll}
\tilde{u}(\omega, 0) 
 = \frac{ {\tilde v}_0(\omega)- i(- i \mu^{1/2}\omega)^{1/2} {\tilde v}_1(\omega) }{2}  +  \frac{ {\tilde v}_0(\omega)+i  (- i \mu^{1/2}\omega)^{1/2}  {\tilde v}_1(\omega)}{2} = {\tilde v}_0(\omega),
  \vspace{2.5mm} \\
\partial_x \tilde{u}(\omega, 0) 
 =  ( i \mu^{1/2} \omega)^{1/2} \left( \frac{ {\tilde v}_0(\omega)- i(- i \mu^{1/2}\omega)^{-1/2} {\tilde v}_1(\omega) }{2}  -  \frac{ {\tilde v}_0(\omega)+i  (- i \mu^{1/2}\omega)^{-1/2} {\tilde v}_1(\omega) }{2}\right) =  {\tilde v}_1(\omega).
\end{array} \]
Meanwhile, based on the relativistic treatment, one-dimensional heat equation
\[ \begin{array}{ll}
\partial_x^2 u -\mu^{1/2}  \partial_t u = 0
\end{array} \]
is written as
\[ \begin{array}{ll}
\partial_x 
\left( \begin{array}{ll}
 u \\
 v
 \end{array}  \right) 
  -
\left( \begin{array}{cc}
0 & I \\
 \mu^{1/2}  \partial_t   & 0
\end{array} \right)
\left( \begin{array}{ll}
 u \\
 v
 \end{array}  \right)
 = 0. \vspace{1.5mm}
\end{array} \]
Let a linear operator ${\mathcal A}$ be defined by
\[ \begin{array}{ll}
{ \mathcal A} = 
\left( \begin{array}{cc}
0 & I \\
\mu^{1/2}  \partial_t   & 0
\end{array} \right)
 \end{array} \]
in $L^2(-\infty,\infty) \times L^2(-\infty,\infty)$, and the domain space of ${\mathcal A}$ be
\[ \begin{array}{ll} 
D({\mathcal A})= H^1(-\infty,\infty) \times L^2(-\infty,\infty), 
\end{array} \]
where $H^1(-\infty,\infty) = \{ u(t) \in L^2(-\infty,\infty);   ~u'(t) \in L^2(-\infty,\infty), ~ u(0) = 0  \}$ is a Sobolev space.
The Fourier transform means that the diagonalization of ${\mathcal A} $ is equal to
\[ \begin{array}{ll}
\tilde{\mathcal A} = 
\left( \begin{array}{cc}
(\mu^{1/2}  \partial_t )^{1/2}  & 0 \\
0  & - ( \mu^{1/2}  \partial_t )^{-1/2} 
\end{array} \right).
 \end{array} \]
In this context the master equation of the problem (\ref{heat2}) is reduced to the abstract evolution equation
\begin{equation} \begin{array}{ll}
\partial_x 
\left( \begin{array}{ll}
 {\tilde u} \\
 {\tilde v}
 \end{array}  \right) 
  -
\tilde{\mathcal A}
\left( \begin{array}{ll}
 {\tilde u} \\
 {\tilde v}
 \end{array}  \right)
 = 0
\end{array} \label{abs} \end{equation}
in $L^2(-\infty,\infty) \times L^2(-\infty,\infty)$, where 
\[  \begin{array}{ll}
\tilde{u}=
 \frac{u - i(- i \mu^{1/2}\omega )^{-1/2} v }{2}  
\end{array} \]
and
\[  \begin{array}{ll}
\tilde{v}=
 \frac{ u +i  (- i \mu^{1/2}\omega)^{-1/2} v }{2}.
\end{array} \]
It suggests that the evolution operator of Eq.~(\ref{heat2}) is generated by
\begin{equation}  \begin{array}{ll}
\pm   (\mu^{1/2} \partial_t)^{1/2}, 
\end{array} \end{equation}
so that it is sufficient to show $\pm   (  \mu^{1/2} \partial_t)^{1/2}$ as the infinitesimal generator.
Note that the operator $\tilde{\mathcal A}$ is not necessarily a generator of analytic semigroup, because the propagation of singularity should be different if the evolution direction is different.
Consequently the existence of semigroup for (\ref{abs}) in the $x$-direction is reduced to show $\pm (\mu^{1/2} \partial_t)^{1/2}$ as the infinitesimal generator in $L^2(-\infty, \infty)$.
In the following, the property of $\mu^{1/2} \partial_t$ is discussed at first, and the fractional power of $ (  \mu^{1/2} \partial_t)^{1/2}$ is studied in the next.

\begin{lemma}
The operator $\mu^{1/2} \partial_t$ with the domain $H^1(-\infty, \infty)$ is the infinitesimal generator in $L^2(-\infty, \infty)$.
\label{lemma1} \end{lemma}

\begin{proof}
Let $\lambda$ be a complex number satisfying ${\rm Re }\lambda >0$.
First the existence of $(\lambda - \mu^{1/2} \partial_t)^{-1}$ is examined.
Let $f(t)$ be included in $L^2 (-\infty,\infty)$. 
Because
\begin{equation} \begin{array}{ll}
(\lambda - \mu^{1/2} \partial_t) u = f
\end{array} \label{iheq} \end{equation}
in one-dimensional interval $(-\infty,\infty)$ is a first-order ordinary differential equation with a constant coefficient, and the global-in-$t$ solution necessarily exists for a given $u(0)=u_0 \in {\mathbb C}$.
That is, $\lambda / \mu^{1/2}$ is included in the resolvent set of $\partial_t$ for an arbitrary complex number $\lambda$, so that $(\lambda - \mu^{1/2} \partial_t)^{-1}$ is concluded to be well-defined in $H^1(-\infty,\infty)$.

Second the resolvent operator $(\lambda - \mu^{1/2} \partial_t)^{-1}$ is estimated from the above.
Since $\lambda / \mu^{1/2}$ is included in the resolvent set of $\partial_t$,
it is readily seen that $(\lambda - \mu^{1/2} \partial_t)^{-1}$ is a bounded operator on $L^2(-\infty, \infty)$.
More precisely, let us consider Eq.~(\ref{iheq}) being equivalent to
\begin{equation}  \begin{array}{ll}
 \partial_t u(t) =   (\lambda / \mu^{1/2}) u(t) - f(t) / \mu^{1/2}.
\end{array}  \label{sim1} \end{equation}
If the inhomogeneous term satisfies $f(t) \in L^2 (-\infty,\infty)$, 
\begin{equation} \begin{array}{ll}
u(t) = 
- \frac{1}{\mu^{1/2}} \int_t^{\infty}  \exp \left( \frac{- \lambda }{ \mu^{1/2} }  (s-t) \right) f(s)  ds
\end{array} \label{intrep} \end{equation}
satisfies Eq.~(\ref{sim1}).
According to the Schwarz inequality,
\[ \begin{array}{ll}
\int_{-\infty}^{\infty} |u(t)|^2 dt  
 = \int_{-\infty}^{\infty} \left| \frac{1}{\mu^{1/2}} \int_t^{\infty}  \exp \left( \frac{- \lambda }{ \mu^{1/2} }  (s-t) \right) f(s)  ds  \right|^2 dt      \vspace{2.5mm}  \\
\quad \le \frac{1}{\mu} \int_{-\infty}^{\infty} \left\{ \int_t^{\infty} 
  \exp \left( \frac{- {\rm Re} \lambda }{ 2 \mu^{1/2} }  (s-t) \right) \right. 
  \left. \exp \left( \frac{- {\rm Re} \lambda }{ 2 \mu^{1/2} }  (s-t) \right) | f(s) |  ds  \right\}^2 dt      \vspace{2.5mm}  \\
\quad \le \frac{1}{\mu} \int_{-\infty}^{\infty}   \int_t^{\infty} 
  \exp \left( \frac{-  {\rm Re} \lambda }{ \mu^{1/2} }  (s-t) \right) ds 
 ~ \int_t^{\infty}  \exp \left( \frac{-  {\rm Re} \lambda }{ \mu^{1/2} }  (s-t) \right) | f(s) |^2  ds ~ dt
\end{array} \]
is obtained, and the equality
\[ \begin{array}{ll}
\int_{t}^{\infty}  \exp \left( \frac{-  {\rm Re} \lambda }{ \mu^{1/2} }  (s-t) \right) ds 
= \int_{0}^{\infty}  \exp \left( \frac{- {\rm Re} \lambda }{ \mu^{1/2} } s \right) ds 
=  \frac{ \mu^{1/2} }{  {\rm Re} \lambda }
\end{array} \]
is positive if ${\rm Re}\lambda >0$ is satisfied.
Its application leads to
\[ \begin{array}{ll}
\int_{-\infty}^{\infty} |u(t)|^2 dt    
 \le \frac{1}{\mu} \frac{ \mu^{1/2} }{  {\rm Re} \lambda }   \int_{-\infty}^{\infty}  
  \int_t^{\infty}  \exp \left( \frac{-  {\rm Re} \lambda }{ \mu^{1/2} }  (s-t) \right) | f(s) |^2  ds ~ dt    \vspace{2.5mm}  \\
\quad = \frac{1}{\mu} \frac{ \mu^{1/2} }{  {\rm Re} \lambda }   \int_{-\infty}^{\infty}  
  \int_{-\infty}^{s}  \exp \left( \frac{-  {\rm Re} \lambda }{ \mu^{1/2} }  (s-t) \right) ~ dt  ~  | f(s) |^2  ~ ds.
\end{array} \]
Further application of the equality
\[ \begin{array}{ll}
\int_{-\infty}^{s}  \exp \left( \frac{-  {\rm Re} \lambda }{ \mu^{1/2} }  (s-t) \right) dt
= \int_{-\infty}^{0}  \exp \left( \frac{  {\rm Re} \lambda }{ \mu^{1/2} } t \right) dt 
=  \frac{ \mu^{1/2} }{  {\rm Re} \lambda }
\end{array} \]
results in
\[ \begin{array}{ll}
\int_{-\infty}^{\infty} |u(t)|^2 dt    \vspace{2.5mm}  
 \le  \frac{1}{\mu} \frac{ \mu^{1/2} }{  {\rm Re} \lambda }   \int_{-\infty}^{\infty}  
 \frac{ \mu^{1/2} }{ {\rm Re} \lambda } ~  | f(s) |^2  ~ ds  
 = \frac{ 1 }{  {\rm Re} \lambda^2 }   \int_{-\infty}^{\infty}   ~  | f(s) |^2  ~ ds,
\end{array} \]
for ${\rm Re}\lambda >0$, and therefore
\[ \begin{array}{ll} 
\| (\lambda I - \mu^{1/2} \partial_t)^{-1} f(t) \|_{L^2(-\infty,\infty)} 
\le \frac{1}{ {\rm Re} \lambda} \| f(t) \|_{L^2(-\infty,\infty)}
\end{array} \]
follows. 
That is, for ${\rm Re}\lambda >0$,
\begin{equation} \label{hyouka-1} \begin{array}{ll}
\|(\lambda I - \mu^{1/2} \partial_t)^{-1} \| \le 1/  {\rm Re} \lambda
\end{array} \end{equation}
is valid.
The surjective property of $\lambda I - \mu^{1/2} \partial_t$ is seen by the unique existence of solutions $u \in L^2(-\infty,\infty)$ for the Cauchy problem of Eq.~(\ref{sim1}).

A semigroup is generated by taking a subset of the complex plane as 
\[ \begin{array} {ll}
\Omega = \left\{ \lambda \in {\mathbb C}; ~ \lambda = {\bar \lambda}  \right\},
\end{array} \]
where $\Omega$ is included in the resolvent set of $\mu^{1/2} \partial_t$.
For $\lambda \in \Omega$, $(\lambda I - \mu^{1/2} \partial_t)^{-1}$ exists, and
\begin{equation} \begin{array}{ll}
\| (\lambda I - \mu^{1/2} \partial_t)^{-n} \| 
 \le 1/({\rm Re} \lambda)^{n}
\end{array} \label{resolves} \end{equation}
is obtained.
Consequently, according to the Lumer-Phillips theorem~\cite{61lumer} for the generation of quasi contraction semigroup, $\mu^{1/2} \partial_t$ is confirmed to be an infinitesimal generator in $L^2(-\infty,\infty)$, and the unique existence of global-in-$x$ weak solution follows.
\end{proof}

The semigroup generated by $\mu^{1/2} \partial_t$ is represented by
\[ \begin{array}{ll}
\left(  \exp (h \mu^{1/2} \partial_t) u \right) (t) = u(t+\mu^{1/2}h), \quad -\infty < h < \infty,
\end{array} \]
so that the group is actually generated by $\mu^{1/2} \partial_t$.
Indeed, the similar estimate as Eq.~(\ref{resolves}) can be obtained for ${\rm Re} \lambda < 0$ with $(\lambda - \mu^{1/2})u=f$ in which the solution $u$ is represented by
\[ \begin{array}{ll}
u(t) = -\frac{1}{\mu^{1/2}} \int_{-\infty}^t \exp \left( \frac{-\lambda}{\mu^{1/2}} (s-t) \right) f(s) ds
\end{array} \]
that should be compared to Eq.~(\ref{intrep}).
The next lemma is valid.

\begin{lemma}
For $0<\alpha <1$, the operator $(\mu^{1/2} \partial_t)^{\alpha}$ is the infinitesimal generator in $L^2(-\infty, \infty)$.
\label{thm1}
\end{lemma}

\begin{proof}
According to Lemma~\ref{lemma1}, $\mu^{1/2} \partial_t$ is the infinitesimal generator in $L^2(-\infty,\infty)$.
For an infinitesimal generator $\mu^{1/2} \partial_t$ in $L^2(-\infty,\infty)$, let the one-parameter semigroup generated by $\mu^{1/2} \partial_t$ be denoted by ${\mathcal V}(x)$.
An infinitesimal generator $\mu^{1/2} \partial_t$ is a closed linear operator in $L^2(-\infty,\infty)$.
Its fractional power
\[ \begin{array} {ll}
(\mu^{1/2} \partial_t)^{\alpha}, \quad 0<\alpha<1
\end{array} \]
has been confirmed to be well-defined by S. Bochner~\cite{49bochner} and R.S. Phillips~\cite{52phillips} as the infinitesimal generator of semigroup (cf.~K. Yosida~\cite{{60yosida}}):
\[ \begin{array} {ll}
W(x) w_0 = \int_0^{\infty} {\mathcal V}(x) w_0 ~ d \gamma(\lambda),
\end{array} \]
for $w_0 \in L^2(-\infty, \infty)$, where $W(x)$ is the semigroup for $x$-direction.
The measure $d \gamma(\lambda) \ge 0$ is defined through the Laplace integral
\[ \begin{array} {ll}
\exp (-tk^{\alpha}) = \int_0^{\infty} \exp(- \lambda k) ~ d \gamma(\lambda),
\end{array} \]
where $t, k>0$ is satisfied.
\end{proof}

By taking $\alpha = 1/2$, $(\mu^{1/2} \partial_t)^{1/2}$ is confirmed to be an infinitesimal generator in $L^2(-\infty, \infty)$.
Because  $\Omega$ is included in the resolvent set of $-(\mu^{1/2} \partial_t)^{1/2}$, it is readily seen that $-(\mu^{1/2} \partial_t)^{1/2}$ is also an infinitesimal generator in $L^2(-\infty, \infty)$.

\subsection{Generalization of the Cole-Hopf transform}
The next theorem follows.

\begin{theorem} \label{thm3}
Let $i$ be an integer satisfying $0 \le i \le n$, and $Y$ be a dense subspace of Banach space $X$.
Let an invertible evolution family $\{ U(x^i, \xi^i)\}_{0 \le x^i,\xi^i \le L}$ be generated by $A(x^i)$ for $0 \le x^i,\xi^i \le L$ in a Banach space $X$.
$U(x^i, \xi^i)$ and $A(x^i)$ are assumed to commute.
For any $u_{\xi^i} \in Y \subset X$, the logarithmic representation
\begin{equation} \label{spatrep3} \begin{array}{ll}
 A(x^i)  U(x^i, \xi^i) u_{\xi^i}  
  =  ( \kappa I +  U(x^i, \xi^i))~ [ \partial_{x^i} {\rm Log} ~ ( U(x^i, \xi^i) + \kappa I) ]  u_{\xi^i} ,
\end{array} \end{equation}
is the generalization of the Cole-Hopf transform in the following sense:
\begin{itemize}
\item the linear equation is not necessarily the heat equation;
\item the spatial dimension of the equation is not necessarily equal to 1;
\item the differential in Eq.~(\ref{spatrep3}) is not necessarily for a spatial variable $x$;
\end{itemize}
where the logarithmic representation is obtained in a general Banach space framework, and $\kappa \ne 0$ is a complex number.
In particular, if $(U(x^i, \xi^i) u_{\xi^i})^{-1}$ exists for a given interval $0 \le x^i,\xi^i \le L$, its normalization
\begin{equation} \label{spatrep5} \begin{array}{ll}
 A(x^i)  
  =  ( \kappa  U(\xi^i, x^i)  +  I)~ [ \partial_{x^i} {\rm Log} ~ ( U(x^i, \xi^i) + \kappa I) ]
\end{array} \end{equation}
defined in $X$ corresponds to $\psi(t,x)$ in Eq.~(\ref{ch-trans}). 
\end{theorem}

\begin{proof}
The abstract case of original Cole-Hopf transform is included in the description of the logarithmic representation (\ref{spatrep3}).
Indeed, let an invertible evolution family $\{ W(x, \xi)\}_{0 \le x,\xi \le L}$ be generated by ${\mathcal A}(x)$ for $0 \le x,\xi \le L$.
According to Lemma~\ref{thm1} the logarithmic representation of relativistic form (\ref{logex2}) is obtained in this case as
\[  \begin{array}{ll}
 {\mathcal A}(x) w_{\xi} =  (I+ \kappa W(\xi, x))~ [ \partial_{x} {\rm Log} ~ (W(x,\xi) + \kappa I) ] w_{\xi}
\end{array} \]
and hence as
\[  \begin{array}{ll}
 {\mathcal A}(x)  W(x, \xi) w_{\xi} 
  =  (\kappa I + W(x, \xi))~ [ \partial_{x} {\rm Log} ~ (W(x,\xi) + \kappa I) ] w_{\xi},
\end{array} \]
using the commutation assumption.
The nonlinear Anzatz $-2 \mu^{-1/2} (\partial_x u(t,x)) ~ u(t,x)^{-1}$ of the Burgers' equation shown in Eq.~(\ref{ch-trans}) is represented by
\begin{equation} \label{spatrep4} \begin{array}{ll}
 -2 \mu^{-1/2}({\mathcal A}(x)  W(x, \xi)) ~  W(x, \xi)^{-1}   \vspace{1.5mm}\\
  = -2 \mu^{-1/2}  (\kappa I + W(x, \xi))~ [ \partial_{x} {\rm Log} ~ (W(x,\xi) + \kappa I) ] ~  W(x, \xi)^{-1}  \vspace{1.5mm}\\
  = -2 \mu^{-1/2}  (\kappa  W(\xi, x) + I)~ [ \partial_{x} {\rm Log} ~ (W(x,\xi) + \kappa I) ].
\end{array} \end{equation}
As in the original derivation of the Cole-Hopf transform, the solution of heat equation $w(t,x) = W(x,\xi) w_{\xi}(t)$ solved along the $x$-axis permits its logarithm function.

The first property arises from the introduction of nonzero $\kappa$ in the general form of the logarithmic representation.
According to the introduction of nonzero $\kappa$, the applicability is significantly increased.
The second property is realized by the abstract nature of the logarithmic representation.
The third property is due to the relativistic treatment.

The similarity between Eq.~(\ref{spatrep5}) and the standard definition of operator norm is clear.
In particular the evolution direction is generalized from $x$ to $x^i$ in Eq.~(\ref{spatrep5}).
\end{proof}

The generalized Cole-Hopf transform (\ref{spatrep3})  shows the usefulness of defining $A(x^i) u_{\xi^i}$ from $ U(x^i, \xi^i) u_{\xi^i}$ for $u_{\xi^i} \in Y \subset X$. 
Indeed, $A(x^i) u_{\xi^i}$ is suggested to be a solution of a certain nonlinear evolution equation for the $x^j$ direction with $i \ne j$.

\section{Concluding remark}
The Burgers' equation (\ref{bgeq}) itself is useful to analyze nonlinear wave phenomena, where the similarity to the other equations such as the Korteweg-de Vries equation~\cite{85korteweg} is noticed.
In particular nonlinear advection term $\psi \partial_x \psi$ appears in those equations.
This term is intriguing enough to be included in the Navier-Stokes equations.
In case of the Cole-Hopf transform, the formation of nonlinearity by the linear solution is formally understood by the Leibnitz rule:
\begin{equation} \label{leib} \begin{array}{ll}
(u v^{-1})' = (u'v - v' u)v^{-2},
\end{array} \end{equation}
in the basic calculus, where the notation $'$ denotes the differentiation along the $t$-axis.
This formula is essentially valid to functional analysis under the suitable assumption, where the commutation between $u$ and $v$, and the existence of $u^{-1}$ and $v^{-1}$ are assumed in this discussion.
Let $u$ and $v$ be the functions of $t$, $x$, $y$, and so on.
In the present discussion, $v$ and $-u v^{-1}$ mean the linear and nonlinear solutions respectively.
\begin{equation} \label{leib2} \begin{array}{ll}
(-u v^{-1})' =  [ v' v^{-1} -  u'  u^{-1} ] (-u v^{-1})
\end{array} \end{equation}
follows.
Both $v' v^{-1}$ and $u' u^{-1}$ are regarded as the logarithmic derivative.

The emergence of the nonlinearity can be understood in a concrete setting.
Indeed, in case of the Cole-Hopf transform, $u = - \partial_x v$ leading to $u' = - (\partial_x v)'$ is valid, then 
\[ \begin{array}{ll}
 u' u^{-1} =     (\partial_x v)'   (\partial_x v)^{-1}  
\end{array} \]
results in
\[ \begin{array}{ll}
(-u v^{-1})' =  [ v' v^{-1}   -   (\partial_x v)'   (\partial_x v)^{-1}] (-u v^{-1}),  
\end{array} \]
where the term $ (\partial_x v)'   (\partial_x v)^{-1}$ is also regarded as the logarithmic derivative of unknown function.
In a symmetric fashion, a specific treatment for the Cole-Hopf transform $v' = \partial_x^2 v$ leads to
\[ \begin{array}{ll}
\partial_x (-u v^{-1}) 
 =  [  (\partial_x v) v^{-1} - (\partial_x^2 v) (\partial_x v)^{-1} ] (-u v^{-1})    
\end{array} \]
and therefore
\[ \begin{array}{ll}
 (-u v^{-1})  \partial_x (-u v^{-1}) 
 =   (\partial_x v) v^{-1} (-u v^{-1})^2   - v' v^{-1} (-uv^{-1})       
\end{array} \]
are true under the commutation between $(-uv^{-1})$ and $\partial_x (-u v^{-1}) $.
This is the key to the nonlinear emergence for the Cole-Hopf transform; $v'v^{-1}$ of Eq.~(\ref{leib2}) generates a nonlinearity.
Consequently the substitution clarifies
\begin{equation} \begin{array}{ll}
(-u v^{-1})'  \vspace{1.5mm} \\
\quad     =   (\partial_x v) v^{-1} (-u v^{-1})^2 -  (-u v^{-1})  \partial_x (-u v^{-1}) 
 -   (\partial_x v)'   (\partial_x v)^{-1}  (-u v^{-1}).  
\end{array} \end{equation}
It shows that the unknown function $-u v^{-1}$ satisfies a nonlinear equation in which nonlinear advection term $ (-uv^{-1}) \partial_x (-u v^{-1})$ appears.
In particular $ (\partial_x^2 v) (\partial_x v)^{-1}$ itself is a logarithmic representation that is structurally equal to $(\partial_x u^2) u^{-2} /2$ under the commutation, because 
\[ \begin{array}{ll}
\partial_x u^2    
=-\partial_x (u  \partial_x v)   
=-(\partial_x u)  (\partial_x v)   - u  \partial_x^2 v   
 =  (\partial_x v   - u)  \partial_x^2 v   
 =  2 (\partial_x^2 v) (\partial_x v)^{-1} u^2  \\
\end{array} \]
is true by utilizing $u = - \partial_x v$.
This term shows a different kind of nonlinearity directly associated with $u \partial_x u$.

As a result the emergence of nonlinearity, which corresponds to the emergence of nonlinear advection term in case of the Cole-Hopf transform, is formulated in the abstract framework. 
As seen in Eq.~(\ref{leib2}), the logarithmic derivative and therefore the abstract logarithmic representation have been clarified to be essential to the emergence of certain kind of nonlinearity.
In conclusion, for an unknown function $uv^{-1}$ with a known function $v$ satisfying a linear equation, one or both terms in a linear combination of logarithmic representation
\[
 v' v^{-1} -  u'  u^{-1} 
\]
can be a template of the infinitesimal generators of nonlinear semigroups.  
Indeed, as seen in the above,  a part $v' v^{-1}  = (v' u^{-1}) (u v^{-1})$ generates a nonlinearity, if the unknown function is taken as $u v^{-1}$ with $u = -\partial_x v$. 
The mathematical treatment of this kind of nonlinear semigroup is notably linear-like in which the superposition is always valid at the level of $v$. 
Furthermore rigorous nonlinear solutions can be obtained by solving the linear problem.

\end{document}